\documentclass{article}
\usepackage{amssymb,amsmath,amscd}
\usepackage[amsmath,thmmarks]{ntheorem}
\theoremseparator{.}
\theoremsymbol{\rule{1ex}{1ex}}
\theorembodyfont{\upshape}
\newtheorem{definition}{Definition}[section]
\newtheorem{remark}[definition]{Remark}

\newtheorem{example}[definition]{Example}

\newtheorem*{proof}{Proof}

\theorembodyfont{\itshape}

\newtheorem{Proposition}[definition]{Proposition}

\theoremnumbering{Alph}
\newtheorem{inttheorem}{Theorem}

\theoremsymbol{}
\newtheorem{lemma}[definition]{Lemma}
\newtheorem{proposition}[definition]{Proposition}
\newtheorem{theorem}[definition]{Theorem}
\newtheorem{corollary}[definition]{Corollary}
\usepackage[a4paper,hscale=0.8,vscale=0.8,hmarginratio=1:1,includehead,headheight=14pt]{geometry}
\usepackage{fancyhdr}
\usepackage{amssymb,amsmath,amscd,dsfont}
\usepackage[all,cmtip,line]{xy}
\usepackage{slashed}
\usepackage{colonequals}

\newenvironment{tabsection}{}{}

\newcommand{\op}[1]{\ensuremath{\operatorname{#1}}}

\newcommand{\cB}{\ensuremath{\mathcal{B}}}
\newcommand{\cC}{\ensuremath{\mathcal{C}}}

\newcommand{\cE}{\ensuremath{\mathcal{E}}}

\newcommand{\cG}{\ensuremath{\mathcal{G}}}
\newcommand{\cH}{\ensuremath{\mathcal{H}}}

\newcommand{\cK}{\ensuremath{\mathcal{K}}}

\newcommand{\cP}{\ensuremath{\mathcal{P}}}

\newcommand{\cR}{\ensuremath{\mathcal{R}}}

 \newcommand{\R}{\ensuremath{\mathbb{R}}}
 
 \newcommand{\N}{\ensuremath{\mathbb{N}}}

\newcommand{\id}{\ensuremath{\operatorname{id}}}
\newcommand{\pr}{\ensuremath{\operatorname{pr}}}
\newcommand{\ev}{\ensuremath{\operatorname{ev}}}

\newcommand{\cat}[1]{\ensuremath{\mathsf{\mathop{#1}}}}

\newcommand{\Aut}{\ensuremath{\operatorname{Aut}}}

\newcommand{\Diff}{\ensuremath{\operatorname{Diff}}}

\newcommand{\Hom}{\ensuremath{\operatorname{Hom}}}

\newcommand{\res}{\ensuremath{\operatorname{res}}}

\newcommand{\se}{\ensuremath{\nobreak\subseteq\nobreak}}
\newcommand{\from}{\ensuremath{\nobreak\colon\nobreak}}
\renewcommand{\to}{\ensuremath{\nobreak\rightarrow\nobreak}}
\newcommand{\toto}{\ensuremath{\nobreak\rightrightarrows\nobreak}}

\newcommand{\coloneq}{\colonequals}
\DeclareMathOperator{\A}{\Sigma}

\newcommand\opn{\ensuremath{\mathrel{\mathpalette\opncls\circ}}}

\newcommand{\opncls}[2]{
  \ooalign{$#1\subseteq$\cr
  \hidewidth\raisefix{#1}\hbox{$#1{\stylefix{#1}#2}\mkern2mu$}\cr}}
\def\raisefix#1{
  \ifx#1\displaystyle
    \raise.39ex
  \else
    \ifx#1\textstyle
      \raise.39ex
    \else
      \ifx#1\scriptstyle
        \raise.275ex
      \else
        \raise.150ex
      \fi
    \fi
  \fi
}
\def\stylefix#1{
  \ifx#1\displaystyle
    \scriptstyle
  \else
    \ifx#1\textstyle
      \scriptstyle
    \else
      \ifx#1\scriptstyle
        \scriptscriptstyle
      \else
        \scriptscriptstyle
      \fi
    \fi
  \fi
}
\DeclareFontFamily{U}{mathx}{\hyphenchar\font45}
\DeclareFontShape{U}{mathx}{m}{n}{
      <5> <6> <7> <8> <9> <10>
      <10.95> <12> <14.4> <17.28> <20.74> <24.88>
      mathx10
      }{}
\DeclareSymbolFont{mathx}{U}{mathx}{m}{n}
\DeclareFontSubstitution{U}{mathx}{m}{n}
\DeclareMathAccent{\widecheck}{0}{mathx}{"71}
\DeclareMathAccent{\wideparen}{0}{mathx}{"75}
\usepackage{graphicx}%
\usepackage{mathtools}
\usepackage[all,cmtip,line]{xy}%
\CompileMatrices%
\newcommand{\xyhookrightarrow}[1]{\ar@{}[r]|-*[@]{\xhookrightarrow{\hphantom{\hspace{#1}}}}}
\newcommand{\xyhookrrightarrow}[1]{\ar@{}[rr]|-*[@]{\xhookrightarrow{\hphantom{\hspace{#1}}}}}
\usepackage[normalem]{ulem}
\usepackage{enumitem}%
\setlist[enumerate]{label={\alph*})}%
\usepackage{hyperref}%
\hypersetup{
 urlcolor = red,
 colorlinks = true,
 linkcolor = blue,
 citecolor = blue,
 linktocpage = true,
 pdftitle = {Functorial aspects of the reconstruction of Lie groupoids from their bisections},
 pdfauthor = {Alexander Schmeding, Christoph Wockel},
 bookmarksopen = true,
 bookmarksopenlevel = 1,
 unicode = true,
 hypertexnames =false 
}%
\usepackage[final]{showlabels}
\newcommand{\Bis}{\ensuremath{\op{Bis}}}
\newcommand{\eBis}{\ensuremath{\op{\overline{Bis}}}}
\newcommand{\Bisf}[1]{\Bis_{#1}}
\newcommand{\Lf}{\ensuremath{\mathbf{L}}}
\newcommand{\Stab}[1]{\ensuremath{K_{#1}}}
\newcommand{\Vtx}[1]{\ensuremath{\op{Vert}_{#1}}}
\newcommand{\Loop}[1]{\ensuremath{\op{Loop}_{#1}}}

\newcommand{\catBLie}{\cat{BanachLieGpds}_{M}}
\newcommand{\catltBLie}{\cat{BanachLieGpds}_{M}^{\op{triv}}}
\newcommand{\catltevBLie}{\cat{BanachLieGpds}_{M}^{\op{triv}, \ev}}

\begin{document}

\begin{flushright}
   {\sf ZMP-HH/15-16}\\
   {\sf Hamburger$\;$Beitr\"age$\;$zur$\;$Mathematik$\;$Nr.$\;$553}\\[2mm]
\end{flushright}

\title{Functorial aspects of the reconstruction of Lie groupoids from their bisections} \author{Alexander
Schmeding\footnote{NTNU Trondheim, Norway
\href{mailto:alexander.schmeding@math.ntnu.no}{alexander.schmeding@math.ntnu.no}
}%
~~and Christoph Wockel\footnote{University of Hamburg, Germany
\href{mailto:christoph@wockel.eu}{christoph@wockel.eu}}}
{\let\newpage\relax\maketitle}

\begin{abstract}
To a Lie groupoid over a compact base $M$, the associated group of bisection
 is an (infinite-dimensional) Lie group. Moreover, under certain circumstances
 one can reconstruct the Lie groupoid from its Lie group of bisections. In the
 present article we consider functorial aspects of these construction principles. The first
 observation is that this procedure is functorial (for morphisms fixing $M$).
 Moreover, it gives rise to an adjunction between the category of Lie groupoids
 over $M$ and the category of Lie groups acting on $M$. In the last section we
 then show how to promote this adjunction to almost an equivalence of
 categories.
\end{abstract}

\medskip


\medskip

(secondary)

\tableofcontents

\section*{Introduction}

\begin{tabsection}
 
 The Lie group structure on bisection Lie groups was constructed in
 \cite{Rybicki02A-Lie-group-structure-on-strict-groups,SchmedingWockel14},
 along with a smooth action of the bisections on the arrow manifold of a Lie
 groupoid. Furthermore, in \cite{SchmedingWockel15} we have established a tight
 connection between Lie groupoids and infinite-dimensional Lie groups. Namely,
 several (re-)construction principles for Lie groupoids from their group of
 bisections and from infinite-dimensional Lie group actions on a compact
 manifold were provided. The present paper considers the categorical aspects of
 these constructions.
 
 It turns out that the construction principles are functorial, i.e.\ they
 induce functors on suitable categories of (possibly infinite-dimensional) Lie
 groupoids and Lie groups. Moreover, we show that the (re-)construction
 functors together with (a suitable version of) the bisection functor discussed
 in \cite{SchmedingWockel14} yield adjoint pairs. Consequently, we deduce
 properties of the bisection functor from these results. These properties are
 interesting in itself to understand the connection between a Lie groupoid and
 its associated Lie group of bisections. Note that we neglect all higher
 structures on the category of Lie groupoids, we will always work with the
 1-category of Lie groupoids over a fixed base manifold.\medskip
 
 We now go into some more detail and explain the main results. Suppose
 $\cG = (G \toto M)$ is a Lie groupoid. This means that $G,M$ are smooth
 manifolds, equipped with submersions $\alpha,\beta\from G\to M$ and an
 associative and smooth multiplication $G\times _{\alpha,\beta}G\to G$ that
 admits a smooth identity map $1\from M\to G$ and a smooth inversion
 $\iota\from G\to G$. Then the bisections $\Bis(\cG)$ of $\cG$ are the sections
 $\sigma\from M\to G$ of $\alpha$ such that $\beta \circ \sigma$ is a
 diffeomorphism of $M$. This becomes a group with respect to
 \begin{equation*}
  (\sigma \star \tau ) (x) \coloneq \sigma ((\beta \circ \tau)(x))\tau(x)\text{ for }  x \in M.
 \end{equation*}
 If $M$ is compact, $G$ is modelled on a metrisable space and the groupoid
 $\cG$ admits an adapted local addition\footnote{The additional structure
 provided by a local addition allow to turn spaces of smooth maps into
 (infinite-dimensional) manifolds. Using this, one can circumvent to a certain
 degree that the category $\cat{Man}$ of (possibly infinite-dimensional)
 manifolds is not cartesian closed.} (cf.\ \cite{michor1980,SchmedingWockel14}), then this group is a
 submanifold of the space of smooth maps $C^{\infty}(M,G)$ and thus a Lie group
 (cf.\ \cite{SchmedingWockel14}). Moreover, the map
 $\beta_* \colon \Bis (\cG) \rightarrow \Diff (M), \sigma \mapsto \beta \circ \sigma$
 is a Lie group morphism which induces a canonical action of the bisections on
 $M$. This Lie group morphism and the associated action are the key ingredients
 to define the (re-)construction functors and suitable versions of the
 bisection functor (cf.\ \cite[Section 3]{SchmedingWockel14}).
 
 Our first aim is to investigate the so called reconstruction functor. Since
 $\Bis (\cG)$ acts on $M$, we can associate an action groupoid
 $\cB (\cG) \coloneq (\Bis (\cG) \ltimes M \toto M)$ to this action. Denote by
 $\cat{LieGroupoids}_M^{\A}$ the category of locally metrisable Lie groupoids which
 admit an adapted local addition. Then the construction of the action groupoid
 gives rise to an endofunctor
 \begin{displaymath}
  \cB \colon \cat{LieGroupoids}_M^{\A} \rightarrow \cat{LieGroupoids}_M^{\A}.
 \end{displaymath}
 In \cite[Theorem 2.21]{SchmedingWockel15} we have already observed that
 certain Lie groupoids $\cG$ can be recovered from the action groupoid
 $\cB (\cG)$ as a quotient in $\cat{LieGroupoids}_M^{\A}$. Furthermore, the
 joint evaluations
 \begin{equation*}
  \ev \colon \Bis (\cG) \times M \rightarrow G,\quad (\sigma , m) \mapsto \sigma (m)
 \end{equation*}
 (for $\cG = (G \toto M)$) induce a natural transformation from $\cB$ to the
 identity.
 
 To understand the relation of the endofunctor $\cB$ to the identity, consider
 the slice category $\cat{LieGroups}_{\Diff (M)}$ of morphism from locally
 metrisable Lie groups into $\Diff (M)$. Then the construction of the bisection
 Lie group induces a pair of functors
 \begin{displaymath}
  \Bis \colon \cat{LieGroupoids}_M^{\A} \rightarrow \cat{LieGroups}_{\Diff (M)}
  \quad \text{ and } \quad \ltimes \colon \cat{LieGroups}_{\Diff (M)}
  \rightarrow \cat{LieGroupoids}_M^{\A} .
 \end{displaymath}
 The functor $\Bis$ sends a groupoid $\cG$ to the morphism
 $\beta_* \colon \Bis (\cG) \rightarrow \Diff (M)$ while $\ltimes$ constructs
 the action groupoid associated to the Lie group morphism
 $K \rightarrow \Diff (M)$ and the natural action of $\Diff (M)$ on $M$. Note
 that by construction we have $\cB= \ltimes \circ \Bis$. We discuss now the
 relation of the three functors and obtain the following result. \medskip
 
 \begin{inttheorem}\label{thm:A}
  The functor $\ltimes$ is left adjoint to the functor $\Bis$. The adjunction
  is given by mapping the morphism $f\from K\ltimes M\to \cG= (G\toto M)$,
  which is given by a smooth map $f\from K\ltimes M\to G$, to the adjoint map
  $f^{\wedge}\from K\to C^{\infty}(M,G)$, which happens to be an element of
  $\Bis(\cG)$. Consequently, the endofunctor $\cB = \ltimes \circ \Bis$ is a
  comonad on $\cat{LieGroupoids}_M^{\A}$.
 \end{inttheorem}
 
 The endofunctor $\cB$ is called \textquotedblleft\emph{re}construction
 functor\textquotedblright \, as one can show that under certain assumptions
 the groupoid $\cG$ is the quotient (in $\cat{LieGroupoids}_M^{\A}$) of
 $\cB (\cG)$ (see Proposition \ref{prop: gpd:quot} for the exact statement).

 We afterwards turn to the question what additional data is needed in addition
 to a Lie group action $H\to \Diff(M)$ in order to build a Lie groupoid that
 has the Lie group $H$ as its bisections. This is what we call the
 ``construction functor''. Observe that the reconstruction functor had a Lie
 groupoid as its input, that we then \emph{re}constructed. So the present
 question is significantly different. In order to simplify matters (for the
 moment), we will consider this question only for transitive (or more precisely
 locally trivial) Lie groupoids. Although locally trivial Lie are always gives
 as gauge groupoids of principal bundles, many morphisms under consideration
 can not be described as morphisms of principal bundles (with fixed structure
 group). However, one can treat these maps as morphisms of (locally trivial)
 Lie groupoids, whence we prefer the groupoid perspective in contrast to the
 principal bundle perspective..
 
 To obtain the construction functor we need to define first the notion of a
 \emph{transitive pair}. Having already fixed the compact manifold $M$, we now
 choose and fix now once and for all an element $m \in M$. From now on we will
 assume that $M$ is connected. A transitive pair $(\theta , H)$ consists of a
 transitive Lie group action $\theta \colon K \times M \rightarrow M$ and a
 normal subgroup $H$ of the $m$-stabiliser $\Stab{m}$ of $\theta$, such that
 $H$ is a regular\footnote{Regularity (in the sense of Milnor) of Lie groups
 roughly means that a certain class of differential equations can be solved on
 the Lie group. This is an essential prerequisite for infinite-dimensional Lie
 theory (see \cite{HGRegLie15} for more information). Up to this point, all
 known Lie groups modelled on suitably complete, i.e.\ Mackey complete, spaces
 are regular.} and co-Banach Lie subgroup of $\Stab{m}$. The guiding example
 here is the transitive pair
 \begin{displaymath}
  \eBis (\cG) \coloneq ( \Bis (\cG) \times M
  \rightarrow M,\, (\sigma,m)\mapsto \beta(\sigma(m)),\quad \{\sigma \in \Bis (\cG) \mid \sigma (m) = 1_m\})
 \end{displaymath}
 induced by the natural action of the bisections of a locally trivial
 Banach-Lie groupoid $\cG$ over a connected manifold $M$. Transitive pairs
 (over $M$ with respect to $m \in M$) together with a suitable notion of
 morphism form a category $\cat{TransPairs}_M$. Note that functor
 sending a transitive pair $(\theta,H)$ to the adjoint morphism
 $\theta^\wedge \colon K \rightarrow \Diff (M)$ induces a forgetful functor
 from $\cat{TransPairs}_M$ to $\cat{LieGroups}_{\Diff (M)}$.
 
 Restricting our attention to the category $\catltBLie$ of locally trivial
 Banach-Lie groupoid over a connected manifold $M$ we obtain the
 \emph{augmented bisection functor}
 \begin{displaymath}
  \eBis \colon \catltBLie \rightarrow \cat{TransPairs}_M.
 \end{displaymath}
 The augmented bisection functor descents via the forgetful functor to
 $\cat{LieGroups}_{\Diff (M)}$ to the functor $\Bis$ (restricted to locally
 trivial Banach-Lie groupoids).
 
 The notion \textquotedblleft transitive pair\textquotedblright is tailored
 in exactly such a way that one can define a \emph{construction functor}
 \begin{displaymath}
  \cR \colon \cat{TransPairs}_M \rightarrow \cat{LieGroupoids}_M^{\A}
 \end{displaymath}
 which associates to a transitive pair a locally trivial Banach Lie groupoid
 (see Section \ref{sect: TP:con} for details). Moreover, the augmented
 bisection functor and the construction functor are closely connected. If we
 apply the functor $\cR$ to $\eBis (\cG)$, we obtain the open subgroupoid of
 $\cG$ of all elements which are contained in the image of a bisection. Hence
 if we assume that $\cG = (G \toto M)$ is a Lie groupoid with bisections
 through each arrow, i.e.\ for all $g \in G$ exists a bisection $\sigma$ with
 $\sigma (\alpha (g)) = g$, then $\cR (\eBis (\cG))$ is isomorphic to $\cG$.
 Denote by $\catltevBLie$ the full subcategory of all locally trivial
 Banach-Lie groupoids with bisections through each arrow. Then our results
 subsume the following theorem. \medskip
 
 \begin{inttheorem}\label{thm:B}
  Let $M$ be a connected and compact manifold. The functor $\cR$ is left
  adjoint to the functor $\eBis$. Furthermore, the functors induce an
  equivalence of categories
  \begin{displaymath}
   \cat{TransPairs}_M \supseteq \eBis (\catltevBLie) \cong \catltevBLie.
  \end{displaymath}
 \end{inttheorem}
 
 The category equivalence in Theorem \ref{thm:B} shows that transitive pairs
 completely describe locally trivial Banach-Lie groupoids which admit sections
 through each arrow. Hence the geometric information can be reconstructed from
 the transitive pair. This result connects infinite-dimensional Lie theory and
 groupoid theory (or equivalently principal bundle theory) and we explore first
 applications of this result in \cite[Section 5]{SchmedingWockel15}. However,
 equally important is the fact that the functors $\cR$ and $\eBis$ form an
 adjoint pair. The adjointness relation allows one to generate a wealth of
 geometrically interesting morphisms from (infinite-dimensional) Lie groups
 into groups of bisections of locally trivial Banach-Lie groupoids.
 
 Finally, we remark that results similar to Theorem \ref{thm:B} can also be
 obtained for a non-connected manifold $M$. However, then one has to deal with
 several technical difficulties in the construction, forcing one to restrict to
 certain full subcategories. We have avoided this to streamline the exposition
 but will briefly comment on these results at the end of Section \ref{sect:
 constr:fun}.
\end{tabsection}

\section{Locally convex Lie groupoids and the bisection functor}
\label{sec:locally_convex_lie_groupoids_and_lie_groups}

\begin{tabsection}
 In this section the Lie theoretic notions and conventions used throughout this paper are recalled. We refer to
 \cite{Mackenzie05General-theory-of-Lie-groupoids-and-Lie-algebroids} for an
 introduction to (finite-dimensional) Lie groupoids and the associated group of
 bisections. The notation for Lie groupoids and their structural maps also
 follows \cite{Mackenzie05General-theory-of-Lie-groupoids-and-Lie-algebroids}.
 However, we do not restrict our attention to finite dimensional Lie groupoids.
 Hence, we have to augment the usual definitions with several comments. Note
 that we will work all the time over a fixed base manifold $M$.
 
 We use the so called Michal-Bastiani calculus (often also called Keller's~$C^r_c$-theory). 
 As the present paper explores the functorial aspects of certain (re)construction principles, 
 details on the calculus (see \cite{hg2002a,BertramGlocknerNeeb04Differential-calculus-over-general-base-fields-and-rings}) are of minor importance and thus omitted. 
 However, there is a chain rule for smooth maps in this setting, whence there exists an associated concept of locally
  convex manifold, i.e., a Hausdorff space that is locally homeomorphic to open
  subsets of locally convex spaces with smooth chart changes. 
  Recall that the associated category $\cat{Man}$ of locally convex (possibly infinite-dimensional) manifolds with smooth maps is not cartesian closed.
  See
  \cite{Wockel13Infinite-dimensional-and-higher-structures-in-differential-geometry,neeb2006,hg2002a}
  for more details.
\end{tabsection}

 \begin{definition}  
  Let $M$ be a smooth manifold. Then $M$ is called \emph{Banach} (or
  \emph{Fr\'echet}) manifold if all its modelling spaces are Banach (or
  Fr\'echet) spaces. The manifold $M$ is called \emph{locally metrisable} if the
  underlying topological space is locally metrisable (equivalently if all
  modelling spaces of $M$ are metrisable).
  \end{definition}
  
\begin{definition}
 Let $\cG = (G \toto M)$ be a groupoid over $M$ with source projection
 $\alpha \colon G \rightarrow M$ and target projection
 $\beta \colon G \rightarrow M$. Then $\cG$ is a \emph{(locally convex and
 locally metrisable) Lie groupoid over $M$} if
 \begin{itemize}
  \item the objects $M$ and the arrows $G$ are locally convex and locally
        metrisable manifolds,
  \item the smooth structure turns $\alpha$ and $\beta$ into surjective
        submersions, i.e., they are locally projections\footnote{This implies
        in particular that the occurring fibre-products are submanifolds of the
        direct products, see \cite[Appendix
        C]{Wockel13Infinite-dimensional-and-higher-structures-in-differential-geometry}.}
  \item the partial multiplication
        $m \colon G \times_{\alpha,\beta} G \rightarrow G$, the object
        inclusion $1 \colon M \rightarrow G$ and the inversion
        $\iota \colon G \rightarrow G$ are smooth.
 \end{itemize}
Define $\cat{LieGroupoids}_{M}$ to be the category of all (locally convex and
 locally metrisable) Lie groupoids over $M$ where the morphisms are given by Lie groupoid morphisms over $\id_M$ (cf.\ \cite[Definition 1.2.1]{Mackenzie05General-theory-of-Lie-groupoids-and-Lie-algebroids}).
\end{definition}

\begin{definition}
Let $\cG$ be a locally convex and locally metrisable Lie groupoid.
The \emph{group of bisections} $\Bis (\cG)$ of $\cG$ is given as the set of sections
 $\sigma \colon M \rightarrow G$ of $\alpha$ such that
 $\beta \circ \sigma \colon M \rightarrow M$ is a diffeomorphism. This is a
 group with respect to
 \begin{equation}\label{eq: BISGP1}
  (\sigma \star \tau ) (x) \coloneq \sigma ((\beta \circ \tau)(x))\tau(x)\text{ for }  x \in M.
 \end{equation}
 The object inclusion $1 \colon M \rightarrow G$ is then the neutral element
 and the inverse element of $\sigma$ is
 \begin{equation}\label{eq: BISGP2}
  \sigma^{-1} (x) \coloneq \iota( \sigma ((\beta \circ\sigma)^{-1} (x)))\text{ for } x \in M.
 \end{equation}
\end{definition}

\begin{tabsection}
 If $\cG$ is a Lie groupoid over a compact base $M$, then
 \cite{SchmedingWockel14} establishes Lie group structure on the group of
 bisections if $\cG$ admits a certain type of local addition. We recall these
 results now. A local addition is a tool used to construct a manifold structure
 on a space of smooth mappings, whence under certain circumstances we can
 circumvent that $\cat{Man}$ is not cartesian closed (see e.g.\
 \cite{Wockel13Infinite-dimensional-and-higher-structures-in-differential-geometry,conv1997,michor1980}).
\end{tabsection}

   \begin{definition}\label{def:local_addition}
   Suppose $N$ is a smooth manifold. Then a \emph{local addition} on $N$ is a
   smooth map $\A\from U\opn TN\to N$, defined on an open neighbourhood $U$ of
   the submanifold $N\se TN$ such that
   \begin{enumerate}
   \item \label{def:local_addition_a} $\pi\times \A\from U\to N\times N$,
   $v\mapsto (\pi(v),\A(v))$ is a diffeomorphism onto an open
   neighbourhood of the diagonal $\Delta N\se N\times N$ and
   \item \label{def:local_addition_b} $\A(0_{n})=n$ for all $n\in N$.
   \end{enumerate}
   We say that $N$ \emph{admits a local addition} if there exist a local addition
   on $N$.
   \end{definition}
  
 To turn the subset $\Bis (\cG)$ of $C^\infty (M,G)$ into a manifold we need to require that the local addition is adapted to the groupoid structure maps.
 
  \begin{definition}(cf.\
  \cite[10.6]{michor1980}) Let $s\from Q\to N$ be a surjective submersion. Then
  a \emph{local addition adapted to $s$} is a local addition
  $\A \from U\opn TQ\to Q$ such that the fibres of $s$ are additively closed
  with respect to $\A$, i.e.\ $\A (v_{q})\in s^{-1}(s(q))$ for all $q\in Q$ and
  $v_{q}\in T_{q}s^{-1}(s(q))$ (note that $s^{-1}(s(q))$ is a submanifold of
  $Q$).
 \end{definition}

\begin{definition}
 A Lie groupoid $\cG=(G\toto M)$ admits an \emph{adapted local
 addition} if $G$ admits a local addition which is adapted to the source
 projection $\alpha$ (or, equivalently, to the target projection $\beta$). 
 Denote by $\cat{LieGroupoids}_{M}^{\A}$ the full subcategory of $\cat{LieGroupoids}_{M}$ whose objects are Lie groupoids over
 $M$ that admit an adapted local addition.
\end{definition}

An adapted local addition exists for every Banach Lie groupoid by \cite[Proposition 3.12]{SchmedingWockel14}, i.e.\ for the category $\catBLie$ of Banach Lie groupoids over $M$ this implies
\begin{displaymath}
 \catBLie \subseteq \cat{LieGroupoids}_{M}^{\A}
\end{displaymath}

Furthermore, we recall the following facts on locally convex Lie groupoids and the Lie group structure of their bisection groups (cf.\ \cite[Section 3]{SchmedingWockel14}):

\begin{proposition}
 \label{prop: A} Suppose $M$ is compact and $\cG=(G  \toto M)$ is a locally
 convex and locally metrisable Lie groupoid over $M$ which admits an adapted
 local addition. Then $\Bis (\cG)$ is a submanifold of $C^{\infty}(M,G)$ and
 this structure turns $\Bis(\cG)$ into a Lie group.
\end{proposition}

This construction gives rise to the bisection functor 

\begin{definition}\label{defn:functorial_interpretation}
 Suppose $M$ is a compact manifold.
 Then the construction of the Lie group of bisections gives rise to a functor 
 \begin{equation*}
  \Bis\from \cat{LieGroupoids}_{M}^{\A}\to \cat{LieGroups}, 
 \end{equation*}
 sending a groupoid $\cG$ to $\Bis (\cG)$ and morphisms $\varphi \colon \cG \rightarrow \cG'$ to $\Bis (\varphi) \colon \Bis (\cG) \rightarrow \Bis (\cG'),\ \sigma \mapsto \varphi \circ \sigma$.
 Here $\cat{LieGroups}$ denotes the category of locally convex Lie groups.
\end{definition}
 
 In the following sections we will study the bisection functor and its relation to functors which we call (re-)construction functors. 
 The leading idea here is that to a certain extend it is possible to reconstruct Lie groupoids from their group of bisections (cf.\ \cite{SchmedingWockel15}).

\section{The bisection functor and the Reconstruction functor}
\label{sec:reconstruction_functor}

\begin{tabsection}
 In this section, we study the reconstruction functor which arises from the canonical action of the bisections on the base manifold.
 We will assume through this section that $M$ is compact and $\cG=(G\toto M)$ is a Lie groupoid in the category $\cat{LieGroupoids}_{M}^{\A}$, i.e.\ $\cG$ is a locally metrisable Lie groupoid which admits an adapted local addition.
\end{tabsection}

\begin{definition}[The bisection action groupoid]\label{defn: bis:act}
 The Lie group $\Bis(\cG)$ has a natural smooth action on $M$, induced by
 $(\beta_{\cG})_{*}\from \Bis(\cG)\to \Diff(M)$ and the natural action of
 $\Diff(M)$ on $M$. This gives rise to the action Lie groupoid
 $\cB(\cG):=\Bis(\cG)\ltimes M$, with source and target projections defined by
 $\alpha_{\cB}(\sigma,m)=m$ and $\beta_{\cB}(\sigma,m)=\beta_{\cG}(\sigma(m))$.
 The multiplication on $\cB(\cG)$ is defined by
 \begin{equation*}
  (\sigma,\beta_{\cG}(\tau(m)))\cdot (\tau,m):=(\sigma \star \tau,m).
 \end{equation*}
 \end{definition}
 
 \begin{remark}\label{rem: bis:act}
 Clearly, any morphism $f\from \cG\to \cH$ of Lie groupoids over $M$ induces a
 morphism $\Bis (f)\times \id_{M}\from \cB(\cG)\to \cB(\cH)$ of Lie groupoids.
 
 Moreover, $\cB(\cG)$ also admits an adapted local addition (for $\alpha_{\cB}$
 and thus also for $\beta_{\cB}$). In fact, this is the case for the Lie group
 $\Bis(\cG)$ and the finite-dimensional manifold $M$ separately, and on
 $\Bis(\cG)\times M$ one can simply take the product of these local additions.
 We may thus interpret $\cB$ as an endofunctor
 \begin{equation*}
  \cB\from \cat{LieGroupoids}_{M}^{\A} \to \cat{LieGroupoids}_{M}^{\A}.
 \end{equation*}
 In addition, the evaluation map $\ev \from \Bis(\cG)\times M\to G, (\sigma , m) \mapsto \sigma (m)$ is a
 morphism of Lie groupoids over $M$
 \begin{equation*}
  \left(\vcenter{  \xymatrix{
  \Bis(\cG)\times M \ar@<1.2ex>[d]^{\beta _{\cB}}\ar@<-1ex>[d]_{\alpha_{\cB}} \\
  M
  }}\right)\xrightarrow{\ev}
  \left(\vcenter{\xymatrix{
  G \ar@<1.2ex>[d]^{\beta _{\cG}}\ar@<-1ex>[d]_{\alpha _{\cG}}\\
  M       
  }}\right),
 \end{equation*}
 which we may interpret as a natural transformation
 $\ev\from \cB\Rightarrow \id$.
 \end{remark}

 In order to understand the categorical structure of the functor $\cB$ and the natural 
 transformation  $\ev\from \cB\Rightarrow \id$ we augment the bisection functor to a functor into a certain slice category.
 
 \begin{definition}
 Define the slice category
 $\cat{LieGroups}_{\Diff(M)}$, in which objects are locally convex and locally
 metrisable Lie groups $K$ that are equipped with a homomorphisms
 $\varphi\from K\to \Diff(M)$. A morphism from $\varphi\from K\to \Diff (M)$ to
 $\varphi'\from K\to \Diff(M)$ is a morphism of locally convex Lie groups
 $\psi\from K\to K'$ such that $\varphi= \varphi' \circ \psi$.
 \end{definition}
 
 Clearly, the functor $\Bis$ induces a functor
 \begin{equation*}
  \Bis\from \cat{LieGroupoids}_{M}^{\A}\to \cat{LieGroups}_{\Diff(M)},\quad (G\toto M)\mapsto (\beta_{*}\from \Bis(\cG)\to \Diff(M)).
 \end{equation*}
 By abuse of notation we will also call this functor the \emph{bisection functor}, as the bisection functor defined in Definition \ref{defn:functorial_interpretation} can be recovered by an application of the forgetful functor $\cat{LieGroups}_{\Diff(M)} \rightarrow \cat{LieGroups}$.
  
 \begin{definition}
 For each object $\varphi\from K\to \Diff(M)$ of $\cat{LieGroups}_{\Diff(M)}$ we
 can construct the action groupoid $K\ltimes M$, which admits an adapted local
 addition by the same argument as for the bisection action groupoid in Remark \ref{rem: bis:act}.
 This gives rise to a functor
 \begin{equation*}
  \ltimes \from \cat{LieGroups}_{\Diff(M)}\to \cat{LieGroupoids}_{M}^{\Sigma}
 \end{equation*}
 which maps a morphism $\psi$ in $\cat{LieGroups}_{\Diff(M)}$ to $\psi \times \id_M$.
\end{definition}

 The construction of $\ltimes$ and the bisection functor above is tailored to yield $\cB=\ltimes \circ \Bis$.
 
\begin{theorem}\label{thm:adjunction}
 The functor $\ltimes$ is left adjoint to the functor $\Bis$. The adjunction is
 given by mapping the morphism $f\from K\ltimes M\to \cG= (G\toto M)$, which is given by a
 smooth map $f\from K\ltimes M\to G$, to the adjoint map
 $f^{\wedge}\from K\to C^{\infty}(M,G)$, which happens to be an element of
 $\Bis(\cG)$.
\end{theorem}

\begin{proof}
 It is clear that $f\mapsto f^{\wedge}$ is natural and injective, since this is
 also the case on the level of morphisms of sets. Thus it remains to show that
 it is surjective and well defined (i.e., $f^\wedge(k)$ is in fact a bisection
 for each $k\in K$ and $k\mapsto f^{\wedge}(k)$ is a homomorphism).
 
 To show surjectivity, assume that some Lie group morphism
 $\psi\from K\to \Bis(\cG)\se C^{\infty}(M,G)$ is given. Then the other adjoint
 $\psi^{\vee}\from K\times M\to G$ is smooth and satisfies
 \begin{multline}\label{eqn4}
  \psi^{\vee}((kk',m))=\psi(kk')(m)=(\psi(k)\star \psi(k'))(m)=\\  \psi(k)(\beta(\psi(k')(m))) \cdot \psi(k')(m)=\psi(k)(k'.m)\cdot\psi(k')(m)=\psi^{\vee}((k,k'.m))\cdot \psi^{\vee}(k',m).
 \end{multline}
 Thus $\psi^{\vee}$ is a morphism of Lie groupoids and we have
 $\psi=(\psi^{\vee})^{\wedge}$.
 
 To verify that $f\mapsto f^{\wedge}$ is well-defined, we first observe that
 for each $k\in K$ the map $m\mapsto \beta( f^{\wedge}(k)(m))=\varphi(k)(m)$ is
 a diffeomorphism by the assumption $\varphi\from K\to \Diff(M)$. It is clear
 that $ f^{\wedge} $ is smooth and that it is a homomorphism of groups follows
 from the adjoint equation to \eqref{eqn4}.
\end{proof}

\begin{remark}\label{rem: adj:counit}
 The counit of the adjunction $\ltimes \dashv\Bis$ is given by the natural
 transformations
 \begin{equation*}
  \ev\from  \Bis(\cG)\ltimes M\to G.
 \end{equation*}
 Indeed, if we take the adjoint $\id^{\vee}$ of the identity
 $\id\from \Bis(\cG)\to \Bis(\cG) $, then we get exactly $\ev$. Likewise, the
 unit of $\ltimes \dashv\Bis$ is given by the natural transformation
 \begin{equation}
  \op{const}\from (\varphi\from K\to \Diff(M))\to ((\beta_{K \ltimes M})_{*}\from \Bis(K\ltimes M)\to \Diff(M)),
 \end{equation}
 which maps an element $k\in K$ to the ``constant'' bisection $m\mapsto (k,m)$
 of $K\ltimes M$.
\end{remark}

\begin{corollary}
 The functor $\cB = \ltimes \circ \Bis$ gives rise to a comonad $(\cB , \ev \colon \cB \Rightarrow \id, \ltimes (\op{const} \Bis)) \colon \cB \Rightarrow \cB \circ \cB)$ (see \cite[3. Proposition 1.6]{BarrWellsTriples05}). 
 Here for $\cG \in \cat{LieGroupoids}_M^{\A}$ the natural transformation $\ltimes (\op{const}_{\Bis (\cG)})$ is given by the formula 
  \begin{displaymath}
   \ltimes (\op{const}_{\Bis (\cG)}) \colon \Bis (\cG) \ltimes M \rightarrow \Bis (\Bis (\cG) \ltimes M) \ltimes M , (\sigma,m) \mapsto (x \mapsto (\sigma ,x) , m)
  \end{displaymath}
\end{corollary}

\begin{corollary}
 The functor
 $\Bis\from \cat{LieGroupoids}_{M}^{\Sigma}\to \cat{LieGroups}_{\Diff(M)}$
 preserves limits. In particular, kernels and pull-backs are preserved.
\end{corollary}

At the end of this section we would like to recall briefly some results from \cite{SchmedingWockel15} to make sense of the term \textquotedblleft reconstruction functor\textquotedblright\
for the endofunctor $\cB$.
The idea behind this is that in certain circumstances, one can recover the groupoid $\cG$ from its bisection action groupoid $\cB (\cG)$ via the natural transformation $\ev$.
To make this explicit, recall the notion of a quotient object in a category.

\begin{remark}
 Each category carries a natural notion of quotient object for an internal equivalence
 relation. If $\cC$ is a category with finite products and $ R\se E\times E$ is
 an internal equivalence relation, then the quotient $E\to E/R$ in $\cC$
 (uniquely determined up to isomorphism) is, if it exists, the coequaliser of
 the diagram
 \begin{equation}\label{eqn1}
  \xymatrix{ R \ar@<.8ex>[r]^{\pr_{1}} \ar@<-.8ex>[r]_{\pr_{2}} & E}.
 \end{equation}
 If, in the case that the quotient exists, \eqref{eqn1} is also the pull-back
 of $E\to E/R$ along itself, then the quotient $E\to E/R$ is called
 \emph{effective} (see \cite[Appendix.1]{Mac-LaneMoerdijk94Sheaves-in-geometry-and-logic} for
 details). 
\end{remark}

In general, $\cG$ will only be a quotient of $\cB (\cG)$ if there are enough bisections, i.e.\ if for every arrow $g \in G$ (with $\cG = (G \toto M)$ there is a bisection $\sigma_g \in \Bis (\cG)$ with $\sigma_g (\alpha (g))=g$.
If $\cG$ is a Lie groupoid with this property, we say that $\cG$ \emph{admits bisections through each arrow}.
A sufficient criterion for this is that the groupoid $\cG$ is source connected, i.e.\ that for each $m \in M$ the source fibre $\alpha^{-1} (m)$ is a connected manifold.  

\begin{lemma}[{{\cite[Theorem 2.14]{SchmedingWockel15}}}]\label{lem: sc:nuff}
 If $\cG$ is a source connected Lie groupoid in $\cat{LieGroupoids}_M^{\A}$ over a compact base $M$, then $\cG$ admits bisections through each arrow.
\end{lemma}

In \cite{SchmedingWockel15} we were then able to prove the following result on groupoids with enough bisections.

\begin{proposition}[{{\cite[Theorem 2.21]{SchmedingWockel15}}}]\label{prop: gpd:quot}
 If $\cG$ is a Lie groupoid with a bisection through each arrow in $G$, e.g.\ $\cG$ is source connected, then the
 morphism $\ev\from \cB(\cG)\to \cG$ is the quotient in $\cat{LieGroupoids}_M^{\A}$ of $\cB(\cG)$ by
 \begin{equation*}
  R=\{(\sigma,m),(\tau,m)\in \Bis(\cG)\times \Bis(\cG)\times M\mid \sigma(m)=\tau(m) \}.
 \end{equation*}
\end{proposition}

\section{Transitive Pairs and the construction functor}\label{sect: TP:con}

In this section we define categories of Lie groups with transitive actions and functors between these categories and categories of Lie groupoids over $M$.
As always, $M$ will be a fixed compact manifold and we shall consider only Lie groupoids in $\cat{LieGroupoids}_M^{\A}$.

\begin{definition}[Transitive pair]\label{defn:
 tgpair}Choose and fix once and for all a point
 $m \in M$.  Let $\theta \colon K \times M \rightarrow M$ be a transitive
 (left-)Lie group action of a Lie group $K$ modelled on a metrisable space and $H$
 be a subgroup of $K$.
 
 Then we call $(\theta , H)$ a \emph{transitive pair} (over $M$ with base point
 $m$) if the following conditions are satisfied:
 \begin{enumerate}[label=({P\arabic*})]
  \item \label{defn: tgpair_1} the map $\theta_m \coloneq \theta (\cdot,m)$ is
        a surjective submersion,
  \item \label{defn: tgpair_2} $H$ is a normal Lie subgroup of the stabiliser
        $\Stab{m}$ of $m$ and this structure turns $H$ into a regular Lie group
        which is co-Banach as a submanifold in $\Stab{m}$.
 \end{enumerate}
 The largest subgroup of $H$ which is a normal subgroup of $K$ is called \emph{kernel} of the transitive pair.\footnote{Note that by \cite[Proposition 4.16]{SchmedingWockel15} every transitive pair admits a kernel (cf.\ the comments before Corollary \ref{cor: ker:cat}). By standard arguments for topological groups, the kernel is a closed subgroup. In general this will not entail that it is a closed Lie subgroup (of the infinite-dimensional Lie group $K$).}
 If the action of $K$ on $M$ is also $n$-fold transitive, then we call
 $(\theta,H)$ an \emph{$n$-fold transitive pair}.
\end{definition}

Transitive pairs have been studied in \cite{SchmedingWockel15} in the context of reconstructions of locally trivial Lie groupoids. 
Conceptually they are closely related to Klein geometries and we refer to loc.cit. for more information on this topic.
In the context of the present paper, we reconsider the construction of a locally trivial Lie groupoid from a transitive pair.
It will turn out that this constructions is functorial on a suitable category of transitive pairs which we introduce now.

\begin{definition}[Category of transitive pairs] 
We define the category $\cat{TransPairs}_{M}$ which has as objects transitive pairs (over $M$ with
 base-point $m$). 
 
 A morphism in $\cat{TransPairs}_{M}$ from
 $(\theta\from K\times M\to M,H)$ to $(\theta'\from K'\times M\to M,H')$ is a
 smooth morphisms of Lie groups $\varphi\from K\to K'$ such that
 $\varphi(H)\se H'$ and $\theta'^\wedge \circ \varphi= \theta^\wedge $. 
 Here $\theta^\wedge \colon K \rightarrow \Diff (m) ,\quad k \mapsto \theta (k,\cdot)$ is the Lie group morphism associated to $\theta$ via the exponential law (see \cite[Theorem
  7.6]{Wockel13Infinite-dimensional-and-higher-structures-in-differential-geometry}).
 \end{definition}

  \begin{remark}\label{rem: forgetful}
   By construction, there is a forgetful functor $\op{For} \colon \cat{TransPairs}_M \rightarrow \cat{LieGroups}_{\Diff (M)}$ which sends a transitive pair $(\theta \colon K \times M \rightarrow M,H)$ to $\theta^\wedge \colon K \rightarrow \Diff (M)$ and a morphism $\varphi \colon (\theta, H) \rightarrow (\theta',H')$ to the underlying Lie group morphism.
  \end{remark}

 \begin{lemma}\label{lem: iso:char}
  A morphism $\varphi \colon (\theta \colon K \times M \rightarrow M, H) \rightarrow (\theta' \colon K' \times M \rightarrow M,H')$ in $\cat{TransPairs}_M$ is an isomorphism if and only if the underlying morphism of Lie groups $\varphi \colon K \rightarrow K'$ is an isomorphism whose inverse maps $H'$ into $H$.
 \end{lemma}
 
 \begin{proof}
  The condition is clearly necessary. 
  Conversely, assume that $\varphi \colon K \rightarrow K'$ is an isomorphism of Lie groups with inverse $\psi \colon K' \rightarrow K$.
  Now $\psi$ maps $H'$ into $H$ and it induces a morphism of transitive pairs since 
  \begin{displaymath}
   \theta^\wedge \circ \psi = \theta'^\wedge \circ \varphi \circ \psi = \theta'^\wedge.
  \end{displaymath}
 \end{proof}

Let us now exhibit two examples of transitive pairs.
To this end recall the notion of a locally trivial Lie groupoid

\begin{definition}
 Let $\cG = (G \toto M)$ be a Lie groupoid. Then we call $\cG$ 
 \emph{locally trivial} if the \emph{anchor map} $(\beta,\alpha) \colon G \rightarrow M \times M, g \mapsto (\beta (g) , \alpha (g))$ is a surjective submersion. 
\end{definition}
 
\begin{example}\label{ex: tgp:bis} 
  \begin{enumerate}
   \item Let $\cG = (G \toto M)$ be a locally trivial Banach-Lie groupoid over a compact manifold $M$.
      Denote by $\Vtx{m}$ the vertex subgroup of the groupoid $\cG$. 
   By \cite[Proposition 3.12]{SchmedingWockel14} $\cG$ admits an adapted local addition, whence $\Bis (\cG)$ becomes a Lie group and \cite[Proposition 3.2 and Proposition 3.4]{SchmedingWockel15} shows that
  \begin{align*}
   \Loop{m} (\cG) &\coloneq \{\sigma \in \Bis (\cG) \mid \sigma (m) \in \Vtx{m} = \alpha^{-1} (m) \cap \beta^{-1} (m)\} ,\\
   \Bisf{m}(\cG) &\coloneq \{\sigma \in \Bis (\cG) \mid \sigma (m) = 1_m\}.
  \end{align*} 
  are regular Lie subgroups of $\Bis (\cG)$ such that $\Bisf{m} (\cG)$ is a normal co-Banach Lie subgroup of $\Loop{m} (\cG)$.
 Note that $\beta \circ \ev \colon \Bis (\cG) \times M \rightarrow M$ is a Lie group action whose $m$-stabiliser is $\Loop{m} (\cG)$.
 Recall from \cite[Example 4.2 a)]{SchmedingWockel15} that the pair $(\beta \circ \ev , \Bisf{m} (\cG))$ is a transitive pair, if $\beta \circ \ev$ is a transitive Lie group action. 
 In particular, $(\beta \circ \ev , \Bisf{m}(\cG))$ is a transitive pair if $M$ is a connected manifold (see Lemma \ref{lem: always:trans}) or for each $g \in G$ there is a bisection $\sigma_g$ with $\sigma_g (\alpha (g)) = g$.\
 \end{enumerate}
The first example motivated the definition of a transitive  pair and is closely tied to the quotient process outlined in the previous section. 
However, one has considerable freedom in choosing the ingredients for such a pair:
 \begin{enumerate}
   \item[b)]  Consider the diffeomorphism group $\Diff (M)$ of a compact and connected manifold $M$.
   Choose a Lie group $B$ modelled on a Banach space and define $K \coloneq \Diff (M) \times B$. 
   Then $K$ becomes a Lie group which acts transitively via $\theta \colon K \times M \rightarrow M , ((\varphi , b),m) \mapsto \varphi (m)$.
   Fix $m \in M$ and observe that $\theta_m$ is a submersion as $\theta_m = \ev_m \circ \pr_1$ and $\ev_m \colon \Diff (M) \rightarrow M$ is a submersion. 
   By construction $\Stab{m} = \Diff_m (M) \times B$ and $H \coloneq K_m$ is a regular (and normal) Lie subgroup of $K_m$.\footnote{Here we have used that $\Diff (M) = \Bis (\cP (M))$ is regular and $\Bisf{m}(\cP(M)) = \Diff_m (M)$. Moreover, $B$ is regular as a Banach Lie group and $\Lf (\Diff (M) \times B) \cong \Lf (\Diff (M)) \times \Lf (B)$.}
   We conclude that $(\ev_m \circ \pr_1 , \Diff_m (M) \times B)$ is a transitive  pair.
   \item[c)] Consider a transitive action $K \times M \rightarrow M$ of a finite-dimensional Lie group on a compact manifold. 
   Then by \cite[Remark 4.3]{SchmedingWockel15} a transitive pair $(\theta, H)$ is given by any normal subgroup $H$ of $\Stab{m}$. 
   As a special case consider the canonical action $SO(3) \times S^2 \rightarrow S^2$ of the special orthogonal group $SO(3)$ on the $2$-sphere (canonically embedded in $\R^{3}$). 
   This action is transitive with abelian stabiliser $SO(2) \cong S^1$.
   Hence, we can choose as $H$ any closed subgroup of $S^1$. In particular. choose as $H$ either $S^1$ or the cyclic subgroups generated by an element with $x^n = 1$ for some $n \in \N$. 
  \end{enumerate}
\end{example}

 Our goal is now to obtain a functor which associates to a transitive pair a locally trivial Lie groupoid. 
 To this end, we have to recall some results from \cite[Section 4]{SchmedingWockel15}.
 
\begin{Proposition}\label{prop: const:pb-R}
 Let $(\theta , H)$ be a transitive  pair. 
  \begin{enumerate}
   \item Then the quotients $K/H$ and $\Lambda_m \coloneq \Stab{m}/H$ are Banach manifolds (such that the quotient maps become submersions). 
 Moreover, the map $\theta_m$ induces a $\Lambda_m$-principal bundle $\pi \colon K/H \rightarrow M, kH \mapsto \theta (k,m)$.
   \item We denote the Gauge groupoid associated to the $\Lambda_m$-principal bundle by  
 \begin{displaymath}
  \cR(\theta, H) \coloneq \left(\vcenter{ \xymatrix{ \frac{K/H \times
  K/H}{\Lambda_m} \ar@<1.2ex>[d]^{\beta _{\cR}}\ar@<-1ex>[d]_{\alpha_{\cR}} \\
  M. }}\right).
 \end{displaymath}
 Its structure maps are given by $\alpha_{\cR} (\langle gH, kH\rangle) = \pi (kH)$ and
 $\beta_{\cR} (\langle gH, kH\rangle) = \pi (gH)$. 
  \end{enumerate}
\end{Proposition}

 \begin{remark}\label{rem: Liegp:loctgpd} Let $(\theta, H)$ be a transitive pair with $\theta \colon K \times M \rightarrow M$.
 \begin{enumerate}
  \item Observe that as $K/H$ is a Banach manifold, the gauge groupoid $\cR(\theta,H)$ is a Banach-Lie groupoid
 and thus $\cR(\theta, H) \in \cat{LieGroupoids}_{M}^{\A}$ by \cite[Proposition
 3.12]{SchmedingWockel14}.
 Moreover, the Gauge groupoid $\cR (\theta, H)$ is source connected if and only if $K/H$ is connected.
  \item Cȟoose a section atlas
 $(\sigma_i,U_i)_{i \in I}$ of $\theta_m \colon K \xrightarrow{\Stab{m}} M$, i.e.\ a family of sections of $\theta_m$ such that $M = \bigcup_{i \in I} U_i$ (since $\theta_m$ is a submersion such an atlas exists). 
 Composing the $\sigma_i$ with the quotient map $p \colon K \rightarrow K/H$, we obtain a section atlas
 $s_i \coloneq p \circ \sigma_i \colon U_i \rightarrow \pi^{-1} (U_i) \subseteq K/H$
 of $\pi \colon K/H \xrightarrow{\Lambda_m} M$ and thus identify the bisections of $\cR (\theta,H)$ with bundle automorphisms
 via 
 \begin{equation}\label{eq: ident:AutBis}
  \Aut (\pi \colon K/H \rightarrow M) \rightarrow \Bis (\cR (\theta,H)),\quad f \mapsto (m \mapsto \langle f(s_i (m)) , s_i (m)\rangle , \text{ if } m\in U_i.
 \end{equation}
  (cf.\ \cite[Example 3.16]{SchmedingWockel14})
 \item Recall that the sections described in b) also induce manifold charts for $\frac{K/H \times K/H}{\Lambda_m}$ via
 \begin{equation}\label{eq: Gau:triv}
  \frac{\pi^{-1} (U_i) \times \pi^{-1} (U_j)}{\Lambda_m} \rightarrow U_i \times U_j \times \Lambda_m ,\quad \langle p_1 , p_2\rangle \mapsto (\pi(p_1) , \pi (p_2), \delta (s_i (\pi (p_1)), p_1) \delta (s_j (\pi(p_2),p_2)^{-1}).
 \end{equation}
 where $\delta \colon K/H \times_\pi K/H \rightarrow \Lambda_m$ is the smooth
 map mapping a pair $(p,q)$ to the element $p^{-1} \cdot q \in \Lambda_m$ which
 maps $p$ to $q$ (via the $\Lambda_m$-right action). 
\end{enumerate}
\end{remark}

 In Section \ref{sec:locally_convex_lie_groupoids_and_lie_groups} we have seen that the category $\catBLie$ of all Banach-Lie groupoids over $M$ is contained in $\cat{LieGroupoids}_M^{\A}$.
 Let us define now now a suitable subcategory of $\catBLie$ together with a functor.
 
 \begin{definition}\label{defn: funct:R}
 \begin{enumerate}
  \item Define the full subcategory $\catltBLie$ of $\catBLie$ whose objects are 
  \begin{align*}
  \op{Ob} \catltBLie &\coloneq \text{locally trivial Banach-Lie groupoids over } M.
  \end{align*}
  \item  Define a functor  
  \begin{align*}
   \cR\from \cat{TransPairs}_{M}\to \catltBLie \subseteq \cat{LieGroupoids}^{\A}_{M}
  \end{align*}
 which maps $(\theta, H)$ to the locally trivial Lie groupoid $\cR (\theta,H)$ (cf.\  Proposition \ref{prop: const:pb-R}) and a morphism $\varphi \colon (\theta,H) \rightarrow (\theta',H')$ in $\cat{TransPairs}_{M}$ to $\cR (\varphi) \colon \cR (\theta,H) \rightarrow \cR (\theta',H'), \langle  kH , gH\rangle \mapsto \langle \varphi(k)H', \varphi (g)H'\rangle$. 
 \end{enumerate}
\end{definition}

 The functor $\cR$ constructs Lie groupoids from transitive pairs. 
 These Lie groupoids are intimately connected to the transitive action of the transitive pair on $M$.
 To see this we recall some results on a natural Lie group morphism induced by first applying $\cR$ and then the bisection functor.
 
 \begin{lemma}[{{see \cite[Lemma 4.11]{SchmedingWockel15}}}]\label{lem:canonical_morphism_into_bisections}
 Let $(\theta, H)$ be a transitive pair. Then the action of $K$ on $K/H$ by
 left multiplication gives rise to a group homomorphsim
 $K \rightarrow \Aut (\pi \colon K/H \xrightarrow{\Lambda_m} M)$. With respect
 to the canonical isomorphism
 $ \Aut (\pi \colon K/H \xrightarrow{\Lambda_m} M) \cong \Bis (\cR (\theta,H))$
 of Lie groups from \eqref{eq: ident:AutBis} this gives rise to the group
 homomorphism
 \begin{equation*} 
  a_{\theta,H} \colon K \rightarrow \Bis (\cR (\theta,H)) ,\quad k \mapsto (x \mapsto \langle k \cdot s_i (x) , s_i (x) \rangle , \text{ for } x \in U_i),
 \end{equation*}
 where $s_i = p_m \circ \sigma_i, i\in I$ are the sections from \ref{rem:
 Liegp:loctgpd} b). Moreover, $a_{\theta,H}$ is smooth and makes the diagram 
 \begin{displaymath}
      \begin{xy}
      \xymatrix{
	K \ar[rr]^{a_{\theta, H}} \ar[rrd]^{\theta^\wedge} & & \Bis (\cR (\theta , H)) \ar[d]^{(\beta_\cR)_*} \\
	& & \Diff (M)
  }
\end{xy}
    \end{displaymath}
 commutative.
\end{lemma}

In general, the Lie group morphism $a_{\theta,H}$ will neither be injective nor surjective (this reflects that the notion of a transitive pairs is quite flexible).
However, one may understand the construction of $a_{\theta,H}$ as a way to obtain an interesting Lie group morphism from a transitive pair into the bisections of suitable locally trivial Lie groupoids over $M$.
Moreover, under some assumptions on the transitive pair, the Lie group morphism $a_{\theta,H}$ lifts to a morphism of transitive pairs. 

\begin{lemma} \label{lem:canonical_morphisms_commute}
 Let $(\theta, H)$ be a transitive pair such that also $(\beta_\cR \circ \ev , \Bisf{m} (\cR (\theta,H)))$ is a transitive pair.
 Then  $a_{\theta,H}$ induces a morphism of transitive pairs $(\theta , H) \rightarrow (\beta_\cR \circ \ev , \Bisf{m} (\cR (\theta,H))$. 
\end{lemma}

\begin{proof} 
 In view of Lemma \ref{lem:canonical_morphism_into_bisections} we only have to prove that for $h \in H$ we have $a_{\theta , H} (h) \in \Bisf{m}(\cR(\theta,H))$. 
 By definition of $a_{\theta,H}$, we have $a_{\theta,H} (h)(m) = \langle h s_i (m) , s_i(m)\rangle$ for the sections $s_i \colon U_i \rightarrow K /H$ of $\pi \colon K/H \rightarrow M$ discussed in Remark \ref{rem: Liegp:loctgpd} b).
 Recall that $\pi$ is induced by the action $\theta$. 
 Thus $m = \pi (s_i (m)) = \theta (s_i (m),m)$ holds and $s_i (m) =k_mH \in \Stab{m} / H $.
 Using the $\Stab{m}/H$-principal bundle structure, the image in the gauge groupoid becomes $a_{\theta,H} (h)(m) = \langle h H , 1_K H\rangle = \langle 1_K H , 1_K H\rangle = 1_{\cR (\theta ,H)} (m)$.
 Summing up, $a_{\theta,H} (H) \subseteq \Bisf{m} (\cR (\theta,H))$ and thus $a_{\theta,H}$ is a morphism of transitive pairs.
\end{proof}

\begin{proposition}\label{prop: tech:atha}
Consider a morphism $\varphi \colon (\theta,H) \rightarrow (\tilde{\theta}, \tilde{H})$ in $\cat{TransPairs}_{M}$.
Then for $a_{\theta,H}$ as in Lemma \ref{lem:canonical_morphism_into_bisections}, the following diagram in $\cat{LieGroups}_{\Diff (M)}$ commutes 
         \begin{displaymath}
         \begin{xy}
      \xymatrix{
	K \ar[rr]^-{a_{\theta, H}} \ar[d]^{\varphi} & & \Bis (\cR (\theta , H)) \ar[d]^{\Bis \circ \cR (\varphi) } \\
	\tilde{K} \ar[rr]^-{a_{\tilde{\theta}, \tilde{H}}} & & \Bis (\cR (\tilde{\theta} , \tilde{H}))
  }
 \end{xy}                                                                                                                                               
    \end{displaymath}
    Hence the family $(a_{\theta , H})_{(\theta, H) \in \cat{TransPairs}_M}$ defines a natural transformation $\op{For} \Rightarrow \Bis \circ \cR$ in $\cat{LieGroups}_{\Diff (M)}$ (where $\op{For}$ is the forgetful functor from Remark \ref{rem: forgetful}).
\end{proposition}

\begin{proof}
 Choose a section atlas $(s_i \colon U_i \rightarrow \pi^{-1} (U_i) \subseteq K/H)_{i \in I}$ of the $\Lambda_m$-principal bundle $\pi \colon K/H \rightarrow M$ as in Remark \ref{rem: Liegp:loctgpd} b). 
 From the definition of a morphism in $\cat{TransPairs}_{M}$ we infer that $\tilde{\theta} \circ (\varphi \times \id_M) = \theta$. 
 Thus $\tilde{\pi} \circ (\varphi \circ \sigma_i) = \tilde{\theta}_m \circ (\varphi \circ \sigma_i) = \theta_m \circ \sigma_i = \pi \circ \sigma_i = \id_{U_i}$ holds and $(\varphi \circ \sigma_i)_{i\in I}$ is a section atlas of the $\tilde{K}_m$-bundle $\tilde{\pi} \colon \tilde{K} \rightarrow M$.
 This section atlas descents to a section atlas of $\tilde{K} / \tilde{H} \xrightarrow{\tilde{\Lambda}_m} M$ which we denote by abuse of notation as $(\varphi \circ s_i)_{i \in I}$.
 In the following we will assume that $x \in U_i$ and the mappings are represented as in Remark \ref{rem: Liegp:loctgpd} b) with respect to the section atlases $(s_i)_{i \in I}$ and $(\varphi \circ s_i)_{i \in I}$.
 Let us now compute the composition $a_{\tilde{\theta},\tilde{H}} \circ \varphi$ given in the diagram.
 Then we obtain for $k \in K$ the formula  
  \begin{displaymath}
   a_{\tilde{\theta},\tilde{H}} \circ \varphi (k) (x) = (x \mapsto \langle \varphi (k) \varphi (s_i(x)) , \varphi (s_i (x))\rangle) = (x \mapsto \langle \varphi (k) \varphi(\sigma_i (x)) \tilde{H} , \varphi (\sigma_i (x)) \tilde{H}\rangle).
  \end{displaymath}
 Now we compute the other composition $(\Bis (\cR (\varphi))) \circ a_{\theta,H}$ of morphisms in the diagram. 
 We obtain 
 \begin{align*}
  \Bis (\cR(\varphi)) \circ a_{\theta,H} (k) &= (x \mapsto \cR(\varphi) ( \langle k\cdot s_i (x),s_i(x)\rangle)) = (x \mapsto \langle \varphi (k \cdot \sigma_i (x))\tilde{H}, \varphi (\sigma_i (x))\tilde{H}\rangle) .
 \end{align*}
 Comparing the right hand sides, the diagram commutes since $\varphi \colon K \rightarrow \tilde{K}$ is a Lie group morphism.  
 
 Lemma \ref{lem:canonical_morphism_into_bisections} implies that the morphisms $a_{\theta,H}$ (for $(\theta, H)$ in $\cat{TransPairs}_M$) are morphisms in $\cat{LieGroups}_{\Diff (M)}$.
 Moreover, we have just seen that the family $(a_{\theta , H})_{(\theta, H) \in \cat{TransPairs}_M}$ is natural and thus defines a natural transformation $\op{For} \Rightarrow \Bis \circ \cR$ in $\cat{LieGroups}_{\Diff (M)}$
\end{proof}

In \cite[Proposition 4.16]{SchmedingWockel15} the kernel of a transitive pair $(\theta,H)$ has been identified with the kernel of the Lie group morphism $a_{\theta,H}$.
Hence the preceeding Proposition shows that morphisms of transitive pairs map kernels of transitive pairs into kernels of transitive pairs.

\begin{corollary}\label{cor: ker:cat}
Let $\varphi \colon (\theta, H) \rightarrow (\tilde{\theta},\tilde{H})$ be a morphism of transitive pairs. 
Then $\varphi$ maps the kernel of the transitive pair $(\theta,H)$ into the kernel of the transitive pair $(\tilde{\theta}, \tilde{H})$.
\end{corollary}

\begin{proof}
Recall from \cite[Proposition 4.16]{SchmedingWockel15} that the kernel of the transitive pair $(\theta,H)$ coincides with the kernel of the Lie group morphims $a_{\theta,H}$.
Hence the assertion follows from the commutative diagram in Proposition \ref{prop: tech:atha}.
\end{proof}

 In the next section we will see that the family $(a_{\theta , H})_{(\theta, H) \in \cat{TransPairs}_M}$ induces a natural transformation in $\cat{TransPairs}_M$. 
 This will establish a close connection between $\cR$ and an augmented version of the bisection functor.
 
 \section{The augmented bisection functor and the locally trivial construction functor}\label{sect: constr:fun}
 
 In this section we define the augmented bisection functor and establish its connection to the functor $\cR$.
 Unless stated explicitly otherwise, we will assume throughout the whole section that $M$ is a compact and \emph{connected} manifold and consider only groupoids in $\catltBLie$.
 Moreover, we choose and fix once and for all $m \in M$.
 Taking $M$ to be a connected manifold allows us to ignore certain technicalities in the definition of the augmented bisection functor (see Definition \ref{defn: aug:bis}).
 With suitable care one can extend the results outlined in this section also for non-connected $M$. 
 However, then one has to restrict the occurring functors to suitable full subcategories of $\catltBLie$ and the statements become a lot more technical. 
 We will briefly comment on this in Remark \ref{rem: non:com}.
 
 We have already seen in Example \ref{ex: tgp:bis} a) that for a locally trivial Lie groupoid which satisfy an additional assumption, the action of the bisections and a certain subgroup yield a transitive pair.
 If $M$ is a connected manifold, no additional assumption is needed and we obtain the following result:
 
 \begin{lemma}\label{lem: always:trans}
  Let $\cG = (G \toto M)$ be a locally trivial Banach-Lie groupoid. 
  Then $(\beta \circ \ev \colon \Bis (\cG) \times M \rightarrow M, \Bisf{m} (\cG))$ is a transitive pair.
 \end{lemma}

 \begin{proof}
  The groupoid $\cG$ is a locally trivial Lie groupoid, whence the image of $\beta_* \colon \Bis (\cG) \rightarrow \Diff (M), \sigma \mapsto \beta \circ \sigma$ contains the identity component $\Diff (M)_0$ of $\Diff (M)$ (cf.\ \cite[Example 3.16]{SchmedingWockel14}).
  Since $M$ is connected, \cite[Corollary 2.17]{SchmedingWockel15} implies that the Lie group action $\gamma \colon \Diff (M) \times M \rightarrow M, (\varphi , m) \mapsto \varphi (m)$ restricts to a transitive action of the group $\Diff (M)_0$. 
  We deduce from $\beta \circ \ev = \gamma \circ \beta_*$ that $\beta \circ \ev$ is transitive, whence $(\beta \circ \ev \colon \Bis (\cG) \times M \rightarrow M, \Bisf{m} (\cG))$ is a transitive pair by Example \ref{ex: tgp:bis} a).  
 \end{proof}

 Note that the Lemma asserts that for a locally trivial Lie groupoid $\cG$ over a connected manifold, the canonical action of the bisection together with $\Bisf{m} (\cG)$ always yields a transitive pair.
 For non-connected base this statement is false (see \cite[Remark 2.18 b)]{SchmedingWockel15} for an example).
 
 \begin{lemma}\label{lem: bisfun:tpair}
  Let $\psi \colon \cG \rightarrow \cG'$ be a morphism in $\catltBLie$. 
  Then $\Bis (\psi) \colon \Bis (\cG) \rightarrow \Bis (\cG')$ in $\cat{LieGroups}_{\Diff (M)}$ induces a morphism of transitive pairs $(\beta_\cG \circ \ev , \Bisf{m} (\cG)) \rightarrow (\beta_{\cG'} \circ \ev, \Bisf{m}(\cG'))$.
 \end{lemma}

 \begin{proof}
   We only have to prove that $\Bis (\psi)$ maps the subgroup $\Bisf{m} (\cG)) = \{\sigma \in \Bis (\cG) \mid \sigma (m) = 1_{\cG} (m) \}$ to $\Bisf{m} (\cG')$.
   However, since $\psi$ is a groupoid morphism over $M$, $\psi \circ 1_\cG = 1_{\cG'}$. 
   Hence $\Bis (\psi ) (\sigma ) (m) = \psi \circ \sigma (m) = 1_{\cG'} (m)$ for all $\sigma \in \Bisf{m} (\cG)$ and $\Bis (\varphi) (\Bisf{m} (\cG)) \subseteq \Bisf{m} (\cG')$.
 \end{proof}

 Using Lemma \ref{lem: always:trans} and Lemma \ref{lem: bisfun:tpair} we can now define the augmented bisection functor.
 
\begin{definition}\label{defn: aug:bis}
 The \emph{augmented bisection functor} on $\catltBLie$ is defined as
 \begin{align*}
   \eBis \colon \catltBLie &\rightarrow \cat{TransPairs}_{M},\\
   \cG \quad  &\mapsto (\beta_\cG \circ \ev \colon \Bis (\cG ) \times M \rightarrow M, \Bisf{m} (\cG)),\\
 (\cG \xrightarrow{\psi} \tilde{\cG}) \quad &\mapsto \Bis (\psi). 
  \end{align*}
\end{definition}
 
 \begin{remark}
  Composing the functor $\eBis$ with the forgetful functor $\op{For}$ from Remark \ref{rem: forgetful} we obtain precisely the functor $\Bis|_{\catltBLie} \colon \catltBLie \rightarrow \cat{LieGroups}_{\Diff (M)}$.
 \end{remark}

 The functor $\eBis$ is closely related to the functor $\cR$ which we discussed in the last section.
 To make this explicit, we recall results from \cite{SchmedingWockel15}.
 
 \begin{example}\label{ex: BisR:inv}
 Let $\cG = (G \toto M)$ be a locally trivial Banach-Lie groupoid.
 We apply $\cR$ to the associated transitive pair $\eBis (\cG) = (\beta \circ \ev, \Bisf{m} (\cG))$ to obtain the gauge groupoid $\cR (\beta \circ \ev, \Bisf{m} (\cG)$.
 This gauge groupoid is related to the locally trivial Lie groupoid $\cG$ via the groupoid homomorphism 
  \begin{equation} \label{eq: iso:chi} 
   \chi_\cG \colon \frac{\Bis (\cG) / \Bisf{m}(\cG) \times \Bis (\cG) / \Bisf{m}(\cG)}{\Loop{m} (\cG)/\Bisf{m}(\cG)} \rightarrow G ,\quad \langle \sigma \Bisf{m}(\cG) , \tau \Bisf{m}(\cG)\rangle \mapsto \sigma (m)\cdot(\tau(m))^{-1}.
  \end{equation}
 Recall from \cite[Lemma 4.21]{SchmedingWockel15} that $\chi_\cG$ induces an isomorphism of the gauge groupoid $\cR (\eBis (\cG)$ onto the open and wide subgroupoid $\ev (\Bis (\cG) \times M) = \{g \in G \mid g =\sigma (\alpha (g)) \text{ for some } \sigma \in \Bis (\cG) \}$ (see \cite[Theorem 2.14]{SchmedingWockel15}).
 
 Recall from Lemma \ref{lem:canonical_morphisms_commute} that the map $a_{\eBis (\cG)} \colon \eBis (\cG) \rightarrow \eBis (\cR (\eBis (\cG)))$ is a morphism of transitive pairs.
 Moreover, it is an isomorphism of transitive pairs with inverse $\eBis (\chi_\cG)$. 
 This follows from the fact that $a_{\eBis (\cG)} \colon \Bis (\cG) \rightarrow \Bis (\cR (\eBis (\cG)))$ is an isomorphism of Lie groups with inverse $\Bis (\chi_\cG)$ by \cite[Lemma 4.21 b)]{SchmedingWockel15}. 
\end{example}
 
 We will now establish that the (Re-)construction functor $\cR$ defined in the last section and the augmented bisection functor form a pair of adjoint functors.
 To this end, we will prove first that the morphisms $\chi_\cG$ induce a natural transformation from $\cR \circ \eBis$ to the identity functor of $\catltBLie$ .
 
 \begin{lemma}\label{lem: chi:funct}
 The family $(\chi_\cG)_{\cG \in \catltBLie}$ of groupoid morphisms over $M$ defined as in \eqref{eq: iso:chi} forms a natural transformation 
 $\chi \colon \cR \circ \eBis \Longrightarrow \id_{\catltBLie}$. 
 \end{lemma}

\begin{proof}
 Consider a Lie groupoid morphism $\psi \colon \cG \rightarrow \cG'$.
 To see that $\chi = (\chi_\cG)_{\cG \in \catltBLie}$ is a natural transformation we pick an element in $\cR (\eBis (\cG))$ and compute 
  \begin{align*}
    \chi_{\cG'} \circ \cR (\Bis (\psi)) (\langle \sigma \Bisf{m}(\cG) , \tau \Bisf{m} (\cG)\rangle) &= \chi_{\tilde{\cG}} (\langle (\psi \circ \sigma) \Bisf{m}(\cG') , (\psi \circ \tau) \Bisf{m} (\cG')\rangle) \\\
      &= \psi (\sigma (m)) \cdot (\psi (\tau (m))^{-1} = \psi (\sigma (m) \cdot (\tau (m))^{-1}) \\ 
      &= \psi \circ \chi_\cG (\langle \sigma \Bisf{m}(\cG) , \tau \Bisf{m} (\cG)\rangle).
  \end{align*}
\end{proof}

Combining Lemma \ref{lem: chi:funct} with the fact that $\Bis (\chi_\cG)$ is an isomorphism, one immediately obtains the following.

 \begin{corollary}\label{cor: nattrans}
  The natural transformation $\eBis (\chi) = (\eBis (\chi_\cG))_{\cG \in \catltBLie}$ is a natural isomorphism 
    \begin{displaymath}
     \eBis (\chi) \colon \eBis \circ \cR \circ \eBis \stackrel{\cong}{\Longrightarrow} \eBis 
    \end{displaymath}
  whose inverse is given by the natural transformation $(a_{\eBis (\cG)})_{\cG \in \catltBLie}$.
 \end{corollary}
 
 Having dealt with natural transformations in $\catltBLie$, we now turn to natural transformations in $\cat{TransPairs}_M$.
 We will lift the natural transformation $(a_{\theta,H})_{\cat{TransPairs}_M}$ in $\cat{LieGroups}_{\Diff (M)}$ (see Proposition \ref{prop: tech:atha}) to a natural transformation in $\cat{TransPairs}_{M}$.
 
\begin{lemma}\label{lem: tech:ath}
 Let $(\theta, H)$ and $(\tilde{\theta}.\tilde{H})$ be transitive pairs and $\varphi \colon (\theta,H) \rightarrow (\tilde{\theta}, \tilde{H})$ be a morphism in $\cat{TransPairs}_{M}$.
 \begin{enumerate}
  \item Then we obtain a commutative diagram in $\cat{TransPairs}_M$
  \begin{equation}\label{eq: diag:atH} \begin{aligned}
         \begin{xy}
      \xymatrix{
	(\theta,H) \ar[rr]^-{a_{\theta, H}} \ar[d]^{\varphi} & & \eBis (\cR (\theta , H)) \ar[d]^{\eBis \circ \cR (\varphi) } \\
	(\tilde{\theta}, \tilde{H}) \ar[rr]^-{a_{\tilde{\theta}, \tilde{H}}} & & \eBis (\cR (\tilde{\theta} , \tilde{H}))
  }
 \end{xy}         \end{aligned}.
    \end{equation}
 Hence the family $(a_{\theta,H})_{(\theta,H) \in \cat{TransPairs}_M}$ forms a natural transformation $\id_{\cat{TransPairs}_M} \Rightarrow \eBis \circ \cR$ in the category $\cat{TransPairs}_M$.   
  \item The map $\cR (a_{\theta,H}) \colon \cR (\theta,H) \rightarrow \cR (\eBis (\cR (\theta ,H)))$ is a Lie groupoid isomorphism with inverse $\chi_{\cR (\theta, H)}$.   
 \end{enumerate}
\end{lemma}

\begin{proof}
\begin{enumerate}
 \item The Lie group map $a_{\theta, H}$ is a morphism of transitive pairs by Lemma \ref{lem:canonical_morphisms_commute}.
 Hence \eqref{eq: diag:atH} makes sense as a diagram in $\cat{TransPairs}_M$. 
 Applying the forgetful functor $\op{For}$ to \eqref{eq: diag:atH}, Proposition \ref{prop: tech:atha} shows that we obtain a commutative diagram in $\cat{LieGroups}_{\Diff (M)}$. 
 We conclude that \eqref{eq: diag:atH} must also be commutative as a diagram in $\cat{TransPairs}_M$.
  
 \item From Example \ref{ex: BisR:inv} we deduce that $\chi_{\cR (\theta, H)}$ will be an isomorphism of Lie groupoids if it is surjective. 
 Hence the assertion will follow if we can prove that $\chi_{\cR (\theta , H)} \circ \cR(a_{\theta,H}) = \id_{\cR (\theta,H)}$.
 Let us use again the section atlas $(s_i)_{i \in I}$ of $\pi\colon K/H \rightarrow M$ from part a). 
 We evaluate the composition of both maps in an element and obtain
  \begin{align*}
   \chi_{\cR (\theta , H) } \circ \cR (a_{\theta, H}) (\langle kH,gH \rangle) & =  \chi_{\cR (\theta,H)} (\langle a_{\theta, H} (k) \Bisf{m} (\cR (\theta, H)) , a_{\theta, H} (g)\Bisf{m} (\cR (\theta, H))\rangle) \\
									      & = a_{\theta,H} (k) (m) \cdot (a_{\theta, H} (g)(m))^{-1} = \langle k\cdot s_i (m) , g \cdot s_i (m)\rangle \\
									      & = \langle kH,gH \rangle .   
  \end{align*}
 Note that this holds since the gauge groupoid operations are given as $\langle kH ,gH \rangle \cdot \langle gH , lH\rangle = \langle kH,lH\rangle$ and $\langle kH ,gH \rangle^{-1} = \langle gH,kH$ (cf.\ \cite[Example 1.1.15]{Mackenzie05General-theory-of-Lie-groupoids-and-Lie-algebroids}).
 Summing up, $\cR (a_{\theta,H})$ is an isomorphism with inverse $\chi_{\cR (\theta , H)}$.
\end{enumerate}
\end{proof}
  
 \begin{corollary}\label{cor: iso:tpair} 
  For a transitive pair $(\theta, H)$ and an arrow $g$ in the Lie groupoid $\cR (\theta, H)$, there is a bisection $\sigma \in \Bis (\cR (\theta,H))$ with $\sigma (\alpha (g)) =g$.
 \end{corollary}
 \begin{proof}
  By Lemma \ref{lem: tech:ath} the morphism $\chi_{\cR (\theta, H)} \colon \cR (\eBis (\cR (\theta,H)) \rightarrow \cR (\theta,H)$ is an isomorphism, whence the assertion follows from \cite[Lemma 4.21 a)]{SchmedingWockel15}
 \end{proof}

\begin{proposition}\label{prop: adj:funct}
 The functor $\cR$ is left adjoint to the functor $\eBis$. 
 The unit of the adjunction is the natural transformation $(a_{\theta,H})_{\cat{TransPairs}_M}$ and the counit of the adjunction is the natural transformation $\chi = (\chi_\cG)_{\cG \in \catltBLie}$.
\end{proposition}

\begin{proof}
 Define for $(\theta,H) \in \cat{TransPairs}_M$ and $\cG \in \catltBLie$ the mapping 
 \begin{displaymath}
  I_{(\theta,H),\cG} \colon \Hom_{\cat{TransPairs}_{M}} ( (\theta , H) , \eBis (\cG)) \rightarrow \Hom_{\catltBLie} (\cR (\theta,H),\cG ),\quad \varphi \mapsto \chi_\cG \circ \cR(\varphi). 
 \end{displaymath}
 We obtain a family of maps which is natural in $\cG$ since the family $\chi = (\chi_\cG)$ is natural by Lemma \ref{lem: chi:funct}. 
 Clearly the family is also natural in the transitive pair $(\theta,H)$.
 To construct an inverse of $I_{(\theta,H),\cG}$ recall from Lemma \ref{lem:canonical_morphisms_commute} that $a_{\theta,H} \colon (\theta,H) \rightarrow (\beta_\cR \circ \ev , \Bisf{m} (\cR (\theta,H)))$ is a morphism of transitive pairs.
 Hence the map 
  \begin{displaymath}
   J_{\cG , (\theta,H)} \colon \Hom_{\catltBLie} (\cR (\theta,H),\cG) \rightarrow  \Hom_{\cat{TransPairs}_{M}} ( (\theta , H) , \eBis (\cG)),\quad \psi \mapsto \eBis (\psi) \circ a_{\theta, H}
  \end{displaymath}
 makes sense. 
 Clearly $J_{\cG , (\theta,H)}$ is natural in $\cG$ and it is natural in $(\theta,H)$ by Lemma \ref{lem: tech:ath} a).
 To see that $J_{\cG , (\theta,H)}$ and $I_{(\theta,H),\cG} $ are mutually inverse, we use the results obtained so far and compute 
  \begin{align*}
    J_{\cG , (\theta,H)}  \circ  I_{(\theta,H),\cG} (\varphi) &\stackrel{\hphantom{\text{Lemma } \ref{lem: tech:ath}\ a)}}{=} \eBis (\chi_\cG \circ \cR (\varphi)) \circ a_{\theta,H} = \eBis (\chi_\cG) \circ  \eBis (\cR (\varphi)) \circ a_{\theta,H} \\
							      &\stackrel{\text{Lemma } \ref{lem: tech:ath}\ a)}{=} \underbrace{\Bis (\chi_\cG) \circ a_{\eBis (\cG)}}_{= \id_{\Bis (\cG)} \text{by Example \ref{ex: BisR:inv}}} \circ \varphi = \varphi , \\
    I_{(\theta,H),\cG} \circ  J_{\cG , (\theta,H)} (\psi)      &\stackrel{\hphantom{\text{Lemma } \ref{lem: tech:ath} \ a)}}{=} \chi_\cG \circ \cR (\eBis (\psi) \circ a_{\theta,H}) = \chi_\cG \circ \cR (\eBis (\psi)) \circ \cR (a_{\theta,H}) 	\\
							      &\stackrel{\text{Lemma } \ref{lem: chi:funct}\hphantom{\ a)}}{=} \psi \circ \underbrace{\chi_{\cR (\theta,H)}  \circ \cR (a_{\theta,H})}_{= \id_{\cR (\theta,H)} \text{ by Lemma \ref{lem: tech:ath} b)}} = \psi .
  \end{align*}
\end{proof}
 
\begin{corollary}
 The functor
 $\eBis\from \catltBLie \to \cat{TransPairs}_{M}$
 preserves limits and the functor $\cR$ preserves colimits. In particular, $\eBis$ preserves kernels and pull-backs and $\cR$ preserves cokernels and push-outs.
\end{corollary}

The results obtained so far can be used to establish a category equivalence between certain subcategories of transitive pairs and locally trivial Banach Lie groupoids.
Let us first define these subcategories.

\begin{definition}
 Define the full subcategory $\catltevBLie$ of $\catltBLie$ whose objects are Lie groupoids $\cG = (G \toto M)$ which admit a bisection through each arrow, i.e.\ $\forall g \in G$ there is $\sigma \in \Bis (\cG)$ such that $\sigma (\alpha (g))= g$.  
\end{definition}

\begin{remark}\label{rem: subcat}
 \begin{enumerate}
  \item Corollary \ref{cor: iso:tpair} asserts that the functor $\cR$ takes its image in $\catltevBLie$. 
  By abuse of notation we will in the following identify $\cR$ with a functor $\cR \colon \cat{TransPairs}_M \rightarrow \catltevBLie$.
  \item If a Banach Lie groupoid $\cG$ (with compact base) is source connected, i.e.\ the fibres of the source projection are connected, then $\cG$ is contained in $\catltevBLie$ by \cite[Theorem 2.14]{SchmedingWockel15}. 
 \end{enumerate}
\end{remark}

Objects in the subcategory $\catltevBLie$ admit \textquotedblleft enough\textquotedblright\ bisections, i.e.\ by Example \ref{ex: BisR:inv} these groupoids can be completely recovered from the transitive pair constructed from their bisections.
Passing to the subcategory $\catltevBLie$, we restrict to all Lie groupoids which can be reconstructed by the functor $\cR$.
Note that Example \ref{ex: BisR:inv} also shows that the functor $\cR \circ \eBis$ associates to every locally trivial Banach-Lie groupoid an open subgroupoid from $\catltevBLie$.
We can view this subgroupoid as a selection of a convenient subobject.
Following \cite[p.56]{MR2240597} this should correspond  on the level of categories to $\catltevBLie$ being a coreflected subcategory of $\catltBLie$.

\begin{proposition}
 The functor $\cE \coloneq \cR \circ  \eBis \colon \catltBLie \rightarrow \catltevBLie$ is right adjoint to the inclusion $\catltevBLie \rightarrow \catltBLie$. In particular, $\catltevBLie$ is a coreflected subcategory of $\catltBLie$.
\end{proposition}

\begin{proof}
 Consider $\cK \in \catltevBLie$ and $\cG \in \catltBLie$ and define 
  \begin{displaymath}
   I_{\cK,\cG} \colon \Hom_{\catltevBLie} (\cK , \cE (\cG) ) \rightarrow \Hom_{\catltBLie} (\cK , \cG) ,\quad \varphi \mapsto \chi_\cG \circ \varphi .
  \end{displaymath}
 The family $(I_{\cK , \cG})_{\cK,\cG}$ is natural in $\cK$ and $\cG$ by Lemma \ref{lem: chi:funct}.
 Let us construct an inverse for $I_{\cK,\cG}$.
 For each Lie groupoid $\cG$ the map $\chi_\cG \colon \cR (\eBis (\cG)) \rightarrow \cG$ restricts by \cite[Lemma 4.21]{SchmedingWockel15} to an isomorphism $\chi_\cG^{res}\colon \cR (\eBis (\cG)) \rightarrow I_\cG \opn \cG$ of Lie groupoids onto its image $I_\cG$.
 Recall from loc.cit.\ that $\chi_\cK$ is an isomorphism for all $\cK$ in $\catltevBLie$, i.e.\ $\chi_\cK = \chi_\cK^{res}$ and thus $I_\cK = \cK$.
 Now for a Lie groupoid morphism $\psi \colon \cK \rightarrow \cG$ we use that every arrow $k$ of $\cK$ is in the image of a bisection to obtain 
  \begin{displaymath}
     \psi (k) = \psi (\sigma_k (\alpha (k)) = \underbrace{\Bis (\psi) (\sigma_k)}_{\in \Bis (\cG)} (\alpha (k)) \in I_\cG                                                                                                                                            
  \end{displaymath}
 Hence the mapping 
 \begin{displaymath}
  J_{\cG , \cK} \colon \Hom_{\catltBLie} (\cK , \cG) \rightarrow \Hom_{\catltevBLie} (\cK , \cE (\cG)) ,\quad \psi \mapsto (\chi_\cG^{res})^{-1} \circ \psi
 \end{displaymath}
 makes sense and the maps $J_{\cG,\cK}$ and $I_{\cK,\cG}$ are mutual inverses.
 Moreover, the family $(J_{\cG , \cK})_{\cG,\cK}$ is natural in $\cK$ and $\cG$ again by Lemma \ref{lem: chi:funct}.
 Thus $\cE$ is right adjoint to the inclusion $\catltevBLie \rightarrow \catltBLie$.
\end{proof}

Finally, we will state the equivalence of categories already alluded to.

\begin{theorem}
 Let $\cC$ be the essential image of $\eBis$.
 Then we obtain an induced equivalence of categories 
  \begin{displaymath}
   \xymatrix{ \cat{TransPairs}_M \ar@<+.7ex>[rr]^-{\cR} & &\catltBLie \ar@<+.7ex>[ll]^-{\eBis}\\
		\cC \ar@{^(->}[u] \ar[rr]^-{\cong} & &   \catltevBLie \ar@{^(->}[u] }
  \end{displaymath}
\end{theorem}

\begin{proof}
 By Remark \ref{rem: subcat} the functor $\cR$ restricts to a functor $\cR^{res} \colon \cC \rightarrow \catltevBLie$ and by construction $\eBis$ restricts to a functor $\eBis^{\res} \coloneq \eBis|_{\catltevBLie}^\cC$.
 From Corollary \ref{cor: iso:tpair} we deduce that $\eBis^{res} \circ \cR^{res} \cong \id_\cC$. 
 Conversely, Lemma \ref{lem: chi:funct} implies that $\cR^{res} \circ \eBis^{res} \cong \id_{\catltevBLie}$.
\end{proof}

\begin{remark}\label{rem: non:com}
 In this section we have so far always assumed that $M$ is connected. 
 Dropping this assumption the assertion of Lemma \ref{lem: always:trans} becomes false, whence the functor $\eBis$ can not be defined on $\catltBLie$ if $M$ is not connected.
 
 However, for $\cG$ in $\catltevBLie$ the pair $(\beta_\cG \circ \ev, \Bisf{m} (\cG))$ still is a transitive pair and thus $\eBis$ makes sense on this subcategory for non-connected $M$.
 Hence if we restrict $\cR$ and $\eBis$ to the full subcategory $\catltevBLie$, one easily proves statements similar to the ones obtained in the present section for non-connected $M$.
 In fact, all proofs given in this section carry over verbatim to the more general setting if we restrict to this subcategory. 
 Note that on $\catltevBLie$ the natural transformation $\chi$ (see Lemma \ref{lem: chi:funct}) induces a natural isomorphism $\cR \circ \eBis|_{\catltevBLie} \stackrel{\cong}{\Longrightarrow} \id_{\catltevBLie}$.
 Using this one can then prove that the functor $\eBis$ restricted to the subcategory reflects isomorphisms. 
\end{remark}
  
  \section*{Acknowledgements} The authors want to thank David Roberts for proposing to view
the functor $\cB$ from Remark \ref{rem: bis:act} as a comonad on a suitable category. The research on this paper was partially supported by the DFG Research
  Training group 1670 \emph{Mathematics inspired by String Theory and Quantum Field Theory},
  the Scientific Network \emph{String Geometry} (DFG project code NI 1458/1-1) and the project
  \emph{Topology in Norway} (Norwegian Research Council project 213458).

 \end{document}